\documentclass[a4paper,leqno,12pt]{amsart}
\usepackage{amsfonts}
\usepackage{graphicx}
\usepackage{amsmath}

\newtheorem{theorem}{Theorem}

\newtheorem{corollary}[theorem]{Corollary}

\newtheorem{definition}{Definition}
\newtheorem{example}{Example}

\newtheorem{lemma}[theorem]{Lemma}

\newtheorem{proposition}[theorem]{Proposition}
\newtheorem{remark}{Remark}

\begin{document}
\title[Characterizations of algebras]{Characterizations of algebras of
rapidly decreasing generalized functions }
\author{C. Bouzar}
\address{Department of Mathematics, Oran-Essenia University, Oran, Algeria}
\email{bouzar@univ-oran.dz; bouzarchikh@hotmail.com}
\author{T. Saidi}
\address{Centre Universitaire de B\'{e}char, Algeria }
\email{saidi\_tb@yahoo.fr }

\begin{abstract}
The well-known characterizations of Schwartz space $\mathcal{S}$ of rapidly
decreasing functions is extended to the algebra $\mathcal{G}_{\mathcal{S}}$
of rapidly decreasing generalized functions and to the algebra $\mathcal{G}_{%
\mathcal{S}}^{\infty }$ of regular rapidly decreasing generalized functions.
\end{abstract}

\maketitle

\section{ Introduction}

The Schwartz space $\mathcal{S}$\ of rapidly decreasing functions on $%
\mathbb{R}^{n}$ and its generalizations$,$ in view of their importance in
many domains of analysis, have been characterized differently by many
authors, e.g. see \cite{GV}, \cite{Kas}, \cite{Ortner}, \cite{SY}, \cite%
{Alvar} and \cite{Groch}. To built a Fourier analysis within the generalized
functions of \cite{Colombeau}, the algebra of rapidly decreasing generalized
functions on $\mathbb{R}^{n},$ denoted $\mathcal{G}_{\mathcal{S}}$, was
first constructed in \cite{Radyno} and recently studied in \cite{Garetto}
and \cite{Del}. The algebra\ of regular rapidly decreasing generalized
functions on $\mathbb{R}^{n},$ denoted $\mathcal{G}_{\mathcal{S}}^{\infty }$%
, is fundamental in the characterization of the local regularity of a
Colombeau generalized function by its Fourier transform and also for
developing a generalized microlocal analysis.

Let

\begin{equation*}
\mathcal{S}^{\ast }=\left\{ f\in \mathcal{C}^{\infty }:\forall \alpha \in
\mathbb{Z}_{+}^{n},\underset{x\in \mathbb{R}^{n}}{\sup }\left\vert \partial
^{\alpha }f\left( x\right) \right\vert <\infty \right\} ,
\end{equation*}%
\begin{equation*}
\mathcal{S}_{\ast }=\left\{ f\in \mathcal{C}^{\infty }:\forall \beta \in
\mathbb{Z}_{+}^{n},\underset{x\in \mathbb{R}^{n}}{\sup }\left\vert x^{\beta
}f\left( x\right) \right\vert <\infty \right\} ,
\end{equation*}%
then, inspired by the work of \cite{Kas}, the authors of \cite{SY}\ proved
the following result :
\begin{equation*}
\mathcal{S=S}^{\ast }\cap \mathcal{S}_{\ast }
\end{equation*}

The aim of this work is to characterize the algebras $\mathcal{G}_{\mathcal{S%
}}$\ and $\mathcal{G}_{\mathcal{S}}^{\infty }$\ in the spirit of the
characterization of the Schwartz space $\mathcal{S}$\ done in \cite{SY}. In
fact we do more, this characterization is given in the general context of
the algebra $\mathcal{G}_{\mathcal{S}}^{\mathcal{R}}\left( \Omega \right) $
of $\mathcal{R}-$regular rapidly decreasing generalized functions on an open
set $\Omega $ of $\mathbb{R}^{n},$\ see \cite{Del} and \cite{Bouzben}.\ The
sixth section of this paper gives such an extension, i.e. the
characterization of the algebra $\mathcal{G}_{\mathcal{S}}^{\mathcal{R}%
}\left( \Omega \right) $ provided $\Omega $ is a box of $\mathbb{R}^{n}.$
The seventh section gives a characterization of $\mathcal{G}_{\mathcal{S}}^{%
\mathcal{R}}\left( \mathbb{R}^{n}\right) $ of $\mathcal{R}-$regular rapidly
decreasing generalized functions on the whole space $\mathbb{R}^{n}$\ using
the Fourier transform. The last section, as corollaries of the results of
the paper, gives the characterizations of the classical algebras $\mathcal{G}%
_{\mathcal{S}}$\ and $\mathcal{G}_{\mathcal{S}}^{\mathcal{\infty }}$.

\section{Regular sets of sequences}

We will adopt the notations and definitions of distributions and Colombeau
generalized functions, see \cite{Hormander} and \cite{Grosser} .

\begin{definition}
\label{Def1}A non void subset $\mathcal{R}$ of $\mathbb{R}_{+}^{\mathbb{Z}%
_{+}}$ is said to be regular, if

For all $\left( N_{m}\right) _{m\in \mathbb{Z}_{+}}\in \mathcal{R}$ and $%
\left( k,k^{\prime }\right) \in \mathbb{Z}_{+}^{2},$ there exists $\left(
N_{m}^{\prime }\right) _{m\in \mathbb{Z}_{+}}\in \mathcal{R}$ such that
\begin{equation}
N_{m+k}+k^{\prime }\leq N_{m}^{\prime }\text{ , }\forall m\in \mathbb{Z}_{+}
\tag{R1}  \label{c1}
\end{equation}

For all $\left( N_{m}\right) _{m\in \mathbb{Z}_{+}}$ and $\left(
N_{m}^{\prime }\right) _{m\in \mathbb{Z}_{+}}$ in $\mathcal{R},$ there
exists $\left( N"_{m}\right) _{m\in \mathbb{Z}_{+}}\in \mathcal{R}$ such
that
\begin{equation}
\max \left( N_{m},N_{m}^{\prime }\right) \leq N"_{m}\mathcal{\ }\text{, }%
\forall m\in \mathbb{Z}_{+}  \tag{R2}  \label{c2}
\end{equation}

For all $\left( N_{m}\right) _{m\in \mathbb{Z}_{+}}$ and $\left(
N_{m}^{\prime }\right) _{m\in \mathbb{Z}_{+}}$ in $\mathcal{R},$ there
exists $\left( N"_{m}\right) _{m\in \mathbb{Z}_{+}}\in \mathcal{R}$ such
that
\begin{equation}
N_{l_{1}}+N_{l_{2}}^{\prime }\leq N"_{l_{1}+l_{2}}\text{ , }\forall \left(
l_{1},l_{2}\right) \in \mathbb{Z}_{+}^{2}  \tag{R3}  \label{c3}
\end{equation}
\end{definition}

\begin{example}
The set $\mathbb{R}_{+}^{\mathbb{Z}_{+}}$ of all positive sequences is
regular.
\end{example}

\begin{example}
The set $\mathcal{A}$ of affine sequences defined by
\begin{equation*}
\mathcal{A=}\left\{ \left( N_{m}\right) _{m\in \mathbb{N}}\in \mathbb{R}%
_{+}^{\mathbb{Z}_{+}}:\exists a\geq 0,\exists b>0,\forall l\in \mathbb{Z}%
_{+},N_{l}\leq al+b\right\}
\end{equation*}%
is regular.
\end{example}

\begin{example}
The set $\mathcal{B}$ of all bounded sequences of $\mathbb{R}_{+}^{\mathbb{Z}%
_{+}}$ is regular.
\end{example}

The notion of regular set is extended to the sets of double sequences.

\begin{definition}
\label{Def2}A non void subset $\widetilde{\mathcal{R}}$ of $\mathbb{R}_{+}^{%
\mathbb{Z}_{+}^{2}}$ is \ said to be regular if

For all $\left( N_{q,l}\right) _{\left( q,l\right) \in \mathbb{Z}%
_{+}^{2}}\in \widetilde{\mathcal{R}}$ and $\left( k,k^{\prime },k^{\prime
\prime }\right) \in \mathbb{Z}_{+}^{3},$ there exists $\left(
N_{q,l}^{\prime }\right) _{\left( q,l\right) \in \mathbb{Z}_{+}^{2}}\in
\widetilde{\mathcal{R}}$ such that
\begin{equation}
N_{q+k,l+k^{\prime }}+k^{\prime \prime }\leq N_{q,l}^{\prime }\text{ , }%
\forall \left( q,l\right) \in \mathbb{Z}_{+}^{2}  \tag{$\widetilde{R}1$}
\label{d1}
\end{equation}

For all $\left( N_{q,l}\right) _{\left( q,l\right) \in \mathbb{Z}_{+}^{2}}$
and $\left( N_{q,l}^{\prime }\right) _{\left( q,l\right) \in \mathbb{Z}%
_{+}^{2}}$ in $\widetilde{\mathcal{R}}$ $,$ there exists $\left(
N"_{q,l}\right) _{\left( q,l\right) \in \mathbb{Z}_{+}^{2}}\in \widetilde{%
\mathcal{R}}$ such that
\begin{equation}
\max \left( N_{q,l},N_{q,l}^{\prime }\right) \leq N"_{q,l}\mathcal{\ }\text{%
, }\forall \left( q,l\right) \in \mathbb{Z}_{+}^{2}  \tag{$\widetilde{R}2$}
\label{d2}
\end{equation}

For all $\left( N_{q,l}\right) _{\left( q,l\right) \in \mathbb{Z}_{+}^{2}}$
and $\left( N_{q,l}^{\prime }\right) _{\left( q,l\right) \in \mathbb{Z}%
_{+}^{2}}$ in $\widetilde{\mathcal{R}}$ $,$ there exists $\left(
N"_{q,l}\right) _{\left( q,l\right) \in \mathbb{Z}_{+}^{2}}\in \widetilde{%
\mathcal{R}}$ such that
\begin{equation}
N_{q_{1},l_{1}}+N_{q_{2},l_{2}}^{\prime }\leq N"_{q_{1}+q_{2},l_{1}+l_{2}}%
\text{ , }\forall \left( q_{1},q_{2},l_{1},l_{2}\right) \in \mathbb{Z}%
_{+}^{4}  \tag{$\widetilde{R}3$}  \label{d3}
\end{equation}
\end{definition}

\begin{example}
The set $\mathbb{R}_{+}^{\mathbb{Z}_{+}^{2}}$ of all positive double
sequences is regular.
\end{example}

\begin{example}
The set $\widetilde{\mathcal{B}}$ of all bounded sequences of $\mathbb{R}%
_{+}^{\mathbb{Z}_{+}^{2}}$ is regular.
\end{example}

The following lemma, not difficult to prove, is needed in the formulation of
the principal theorem of this paper.

\begin{lemma}
Let $\widetilde{\mathcal{R}}$ be a regular subset of $\mathbb{R}_{+}^{%
\mathbb{Z}_{+}^{2}}$, then

$\left( i\right) $ The subset $\mathcal{R}^{0}:=\left\{ N_{.,0}:N\in
\widetilde{\mathcal{R}}\right\} $ is regular in $\mathbb{R}_{+}^{\mathbb{Z}%
_{+}}$.

$\left( ii\right) $ The subset $\mathcal{R}_{0}:=\left\{ N_{0,.}:N\in
\widetilde{\mathcal{R}}\right\} $ is regular in $\mathbb{R}_{+}^{\mathbb{Z}%
_{+}}$.
\end{lemma}

\section{The algebra of\ $\mathcal{R-}$regular bounded generalized functions}

Let
\begin{equation*}
\mathcal{S}^{\ast }\left( \Omega \right) =\left\{ f\in \mathcal{C}^{\infty
}\left( \Omega \right) :\forall \alpha \in \mathbb{Z}_{+}^{n},\underset{x\in
\Omega }{\sup }\left\vert \partial ^{\alpha }f\left( x\right) \right\vert
<\infty \right\} ,
\end{equation*}%
and $\mathcal{R}$ be a regular subset of $\mathbb{R}_{+}^{\mathbb{Z}_{+}},$
if we define

\begin{equation*}
\mathcal{E}_{\mathcal{S}^{\ast }}^{\mathcal{R}}\left( \Omega \right)
=\left\{ \left( u_{\epsilon }\right) _{\epsilon }\in \mathcal{S}^{\ast
}\left( \Omega \right) ^{I}:\exists N\in \mathcal{R},\forall \alpha \in
\mathbb{Z}_{+}^{n},\underset{x\in \Omega }{\sup }\left\vert \partial
^{\alpha }u_{\epsilon }\left( x\right) \right\vert =O\left( \epsilon
^{-N_{\left\vert \alpha \right\vert }}\right) ,\epsilon \longrightarrow
0\right\} ,
\end{equation*}
\begin{equation*}
\mathcal{N}_{\mathcal{S}^{\ast }}^{\mathcal{R}}\left( \Omega \right)
=\left\{ \left( u_{\epsilon }\right) _{\epsilon }\in \mathcal{S}^{\ast
}\left( \Omega \right) ^{I}:\forall N\in \mathcal{R},\forall \alpha \in
\mathbb{Z}_{+}^{n},\underset{x\in \Omega }{\sup }\left\vert \partial
^{\alpha }u_{\epsilon }\left( x\right) \right\vert =O\left( \epsilon
^{N_{\left\vert \alpha \right\vert }}\right) ,\epsilon \longrightarrow
0\right\} ,
\end{equation*}
where $I=\left] 0,1\right] ,$ then the properties of $\mathcal{E}_{\mathcal{S%
}^{\ast }}^{\mathcal{R}}\left( \Omega \right) $ and $\mathcal{N}_{\mathcal{S}%
^{\ast }}^{\mathcal{R}}\left( \Omega \right) $, easy to verify, are given by
the following results.

\begin{proposition}
$\left( i\right) $ The space $\mathcal{E}_{\mathcal{S}^{\ast }}^{\mathcal{R}%
}\left( \Omega \right) $ is a subalgebra of $\mathcal{S}^{\ast }\left(
\Omega \right) ^{I}.$

$\left( ii\right) $ The space $\mathcal{N}_{\mathcal{S}^{\ast }}^{\mathcal{R}%
}\left( \Omega \right) $ is an ideal of $\mathcal{E}_{\mathcal{S}^{\ast }}^{%
\mathcal{R}}\left( \Omega \right) .$

$\left( iii\right) $ We have $\mathcal{N}_{\mathcal{S}^{\ast }}^{\mathcal{R}%
}\left( \Omega \right) =\mathcal{N}_{\mathcal{S}^{\ast }}\left( \Omega
\right) ,$ where
\begin{equation*}
\mathcal{N}_{\mathcal{S}^{\ast }}\left( \Omega \right) =\left\{ \left(
u_{\epsilon }\right) _{\epsilon }\in \mathcal{S}^{\ast }\left( \Omega
\right) ^{I}:\forall m\in \mathbb{Z}_{+},\forall \alpha \in \mathbb{Z}%
_{+}^{n},\underset{x\in \Omega }{\sup }\left\vert \partial ^{\alpha
}u_{\epsilon }\left( x\right) \right\vert =O\left( \epsilon ^{m}\right)
,\epsilon \longrightarrow 0\right\}
\end{equation*}
\end{proposition}

We have also the null characterization of the ideal $\mathcal{N}_{\mathcal{S}%
^{\ast }}\left( \Omega \right) $ provided $\Omega $\ is a box.

\begin{definition}
An open subset $\Omega $ of $\mathbb{R}^{n}$ is said to be a box if
\begin{equation*}
\Omega =\mathbf{I}_{1}\times \mathbf{I}_{2}\times ...\times \mathbf{I}_{n},
\end{equation*}
where each $\mathbf{I}_{i}$ is a finite or infinite open interval in $%
\mathbb{R}$.
\end{definition}

\begin{proposition}
\label{propnull1}Let $\Omega $ be a box, then an element $\left( u_{\epsilon
}\right) _{\epsilon }$ $\in \mathcal{E}_{\mathcal{S}^{\ast }}^{\mathcal{R}%
}\left( \Omega \right) $ belongs to $\mathcal{N}_{\mathcal{S}^{\ast }}\left(
\Omega \right) $ if and only if the following condition is satisfied
\begin{equation}
\forall m\in \mathbb{Z}_{+},\underset{x\in \Omega }{\sup }\left\vert
u_{\epsilon }\left( x\right) \right\vert =O\left( \epsilon ^{m}\right) \text{
},\text{ }\epsilon \longrightarrow 0  \label{5-1}
\end{equation}
\end{proposition}

\begin{proof}
Suppose that $\left( u_{\epsilon }\right) _{\epsilon }$ $\in $ $\mathcal{E}_{%
\mathcal{S}^{\ast }}^{\mathcal{R}}\left( \Omega \right) $ satisfies $\left( %
\ref{5-1}\right) .$ It suffices to show that $\left( \partial
_{i}u_{\epsilon }\right) _{\epsilon }$ satisfies the $\mathcal{N}_{\mathcal{S%
}^{\ast }}\left( \Omega \right) $ estimates for all $i=1,..,n.$ Suppose that
$u_{\epsilon }$ is real valued, in the complex case, we shall carry out the
following calculus separately on its real and imaginary part. Let $m\in
\mathbb{Z}_{+},$ we have to show
\begin{equation*}
\underset{x\in \Omega }{\sup }\left\vert \partial _{i}u_{\epsilon }\left(
x\right) \right\vert =O\left( \epsilon ^{m}\right) ,\text{ }\epsilon
\longrightarrow 0
\end{equation*}%
Since $\left( u_{\epsilon }\right) _{\epsilon }$ $\in $ $\mathcal{E}_{%
\mathcal{S}^{\ast }}^{\mathcal{R}}\left( \Omega \right) ,$ then
\begin{equation}
\exists N\in \mathcal{R},\underset{x\in \Omega }{\sup }\left\vert \partial
_{i}^{2}u_{\epsilon }\left( x\right) \right\vert =O\left( \epsilon
^{-N_{2}}\right) ,\text{ }\epsilon \longrightarrow 0  \label{5-2}
\end{equation}%
Since $\left( u_{\epsilon }\right) _{\epsilon }$ satisfies $\left( \ref{5-1}%
\right) ,$ we have
\begin{equation}
\forall m\in \mathbb{Z}_{+},\underset{x\in \Omega }{\sup }\left\vert
u_{\epsilon }\left( x\right) \right\vert =O\left( \epsilon
^{N_{2}+2m}\right) ,\text{ }\epsilon \longrightarrow 0  \label{5-3}
\end{equation}%
By Taylor's formula, we have
\begin{equation*}
u_{\epsilon }\left( x+\epsilon ^{N_{2}+m}e_{i}\right) =u_{\epsilon }\left(
x\right) +\partial _{i}u_{\epsilon }\left( x\right) \epsilon ^{N_{2}+m}+%
\frac{1}{2}\partial _{i}^{2}u_{\epsilon }\left( x+\theta \epsilon
^{N_{2}+m}e_{i}\right) \epsilon ^{2\left( N_{2}+m\right) },
\end{equation*}%
where $\theta \in \left] 0,1\right[ $ and $\epsilon $ is sufficiently small
as $\Omega $ is a box$.$ It follows that

$\left\vert \partial _{i}u_{\epsilon }\left( x\right) \right\vert \leq
\underset{\left( \ast \right) }{\underbrace{\left\vert u_{\epsilon }\left(
x+\epsilon ^{N_{2}+m}e_{i}\right) \right\vert \epsilon ^{-N_{2}-m}}}+%
\underset{\left( \ast \ast \right) }{\underbrace{\left\vert u_{\epsilon
}\left( x\right) \right\vert \epsilon ^{-N_{2}-m}}}+\underset{\left( \ast
\ast \ast \right) }{\underbrace{\epsilon ^{N_{2}+m}\left\vert \partial
_{i}^{2}u_{\epsilon }\left( x+\theta \epsilon ^{N_{2}+m}e_{i}\right)
\right\vert }}$

From $\left( \ref{5-3}\right) ,$ we have $\left( \ast \right) $ and $\left(
\ast \ast \right) $ are of order $O\left( \epsilon ^{m}\right) ,$ $\epsilon
\longrightarrow 0;$ and from $\left( \ref{5-2}\right) ,$ we have $\left(
\ast \ast \ast \right) $ is of order $O\left( \epsilon ^{m}\right) ,$ $%
\epsilon \longrightarrow 0.$
\end{proof}

\begin{definition}
Let $\mathcal{R}$ be a regular subset of $\mathbb{R}_{+}^{\mathbb{Z}_{+}}$,
the algebra of \ $\mathcal{R-}$regular bounded generalized functions,
denoted $\mathcal{G}_{\mathcal{S}^{\ast }}^{\mathcal{R}}\left( \Omega
\right) ,$ is the quotient algebra
\begin{equation}
\mathcal{G}_{\mathcal{S}^{\ast }}^{\mathcal{R}}\left( \Omega \right) =\frac{%
\mathcal{E}_{\mathcal{S}^{\ast }}^{\mathcal{R}}\left( \Omega \right) }{%
\mathcal{N}_{\mathcal{S}^{\ast }}\left( \Omega \right) }
\end{equation}
\end{definition}

\begin{remark}
When $\mathcal{R}$\ is the set of all positive sequences the algebra $%
\mathcal{G}_{\mathcal{S}^{\ast }}^{\mathcal{R}}\left( \Omega \right) $\ is
denoted by $\mathcal{G}_{L^{\infty }}\left( \Omega \right) $ in \cite{BO}\
and \cite{CHO} as it is constructed on the differential algebra $%
D_{L^{\infty }}(\Omega )$\ of Schwartz \cite{Schwartz}. So it is more
correct to write $\mathcal{G}_{L^{\infty }}^{\mathcal{R}}\left( \Omega
\right) $\ instead of $\mathcal{G}_{\mathcal{S}^{\ast }}^{\mathcal{R}}\left(
\Omega \right) $. The null characterization of negligible elements of $%
\mathcal{G}_{L^{\infty }}\left( \Omega \right) $\ in the case $\Omega =%
\mathbb{R}^{n}$\ is given in \cite{Garetto2}.
\end{remark}

\section{The algebra of $\mathcal{R-}$regular roughly decreasing generalized
functions}

Let\
\begin{equation*}
\mathcal{S}_{\ast }\left( \Omega \right) =\left\{ f\in \mathcal{C}^{\infty
}\left( \Omega \right) :\forall \beta \in \mathbb{Z}_{+}^{n},\underset{x\in
\Omega }{\sup }\left\vert x^{\beta }f\left( x\right) \right\vert <\infty
\right\} ,
\end{equation*}%
and $\mathcal{R}$ be a regular subset of $\mathbb{R}_{+}^{\mathbb{Z}_{+}},$
if we define

\begin{equation*}
\mathcal{E}_{\mathcal{S}_{\ast }}^{\mathcal{R}}\left( \Omega \right)
=\left\{ \left( u_{\epsilon }\right) _{\epsilon }\in \mathcal{S}_{\ast
}\left( \Omega \right) ^{I}:\exists N\in \mathcal{R},\forall \beta \in
\mathbb{Z}_{+}^{n},\underset{x\in \Omega }{\sup }\left\vert x^{\beta
}u_{\epsilon }\left( x\right) \right\vert =O\left( \epsilon ^{-N_{\left\vert
\beta \right\vert }}\right) ,\epsilon \longrightarrow 0\right\} ,
\end{equation*}
\begin{equation*}
\mathcal{N}_{\mathcal{S}_{\ast }}^{\mathcal{R}}\left( \Omega \right)
=\left\{ \left( u_{\epsilon }\right) _{\epsilon }\in \mathcal{S}_{\ast
}\left( \Omega \right) ^{I}:\forall N\in \mathcal{R},\forall \beta \in
\mathbb{Z}_{+}^{n},\underset{x\in \Omega }{\sup }\left\vert x^{\beta
}u_{\epsilon }\left( x\right) \right\vert =O\left( \epsilon ^{N_{\left\vert
\beta \right\vert }}\right) ,\epsilon \longrightarrow 0\right\} ,
\end{equation*}
then the following properties of $\mathcal{E}_{\mathcal{S}_{\ast }}^{%
\mathcal{R}}\left( \Omega \right) $ and $\mathcal{N}_{\mathcal{S}_{\ast }}^{%
\mathcal{R}}\left( \Omega \right) $ are easy to verify.

\begin{proposition}
$\left( i\right) $ The space $\mathcal{E}_{\mathcal{S}_{\ast }}^{\mathcal{R}%
}\left( \Omega \right) $ is a subalgebra of $\mathcal{S}_{\ast }\left(
\Omega \right) ^{I}.$

$\left( ii\right) $ The space $\mathcal{N}_{\mathcal{S}_{\ast }}^{\mathcal{R}%
}\left( \Omega \right) $ is an ideal of $\mathcal{E}_{\mathcal{S}_{\ast }}^{%
\mathcal{R}}\left( \Omega \right) .$

$\left( iii\right) $ We have $\mathcal{N}_{\mathcal{S}_{\ast }}^{\mathcal{R}%
}\left( \Omega \right) =\mathcal{N}_{\mathcal{S}_{\ast }}\left( \Omega
\right) ,$ where
\begin{equation*}
\mathcal{N}_{\mathcal{S}_{\ast }}\left( \Omega \right) =\left\{ \left(
u_{\epsilon }\right) _{\epsilon }\in \mathcal{S}_{\ast }\left( \Omega
\right) ^{I}:\forall m\in \mathbb{Z}_{+},\forall \beta \in \mathbb{Z}%
_{+}^{n},\underset{x\in \Omega }{\sup }\left\vert x^{\beta }u_{\epsilon
}\left( x\right) \right\vert =O\left( \epsilon ^{m}\right) ,\epsilon
\longrightarrow 0\right\}
\end{equation*}
\end{proposition}

The following proposition characterizes $\mathcal{N}_{\mathcal{S}_{\ast
}}\left( \Omega \right) $

\begin{proposition}
\label{prop.null}Let $\left( u_{\epsilon }\right) _{\epsilon }$ $\in
\mathcal{E}_{\mathcal{S}_{\ast }}^{\mathcal{R}}\left( \Omega \right) ,$ then
$\left( u_{\epsilon }\right) _{\epsilon }$ $\in \mathcal{N}_{\mathcal{S}%
_{\ast }}\left( \Omega \right) $ if and only if the following condition is
satisfied
\begin{equation}
\forall m\in \mathbb{Z}_{+},\underset{x\in \Omega }{\sup }\left\vert
u_{\epsilon }\left( x\right) \right\vert =O\left( \epsilon ^{m}\right) \text{
},\text{ }\epsilon \longrightarrow 0  \label{Ns Garetto}
\end{equation}
\end{proposition}

\begin{proof}
Suppose that $\left( u_{\epsilon }\right) _{\epsilon }$ $\in $ $\mathcal{E}_{%
\mathcal{S}_{\ast }}^{\mathcal{R}}\left( \Omega \right) $ satisfies $\left( %
\ref{Ns Garetto}\right) ,$ since $\left( u_{\epsilon }\right) _{\epsilon }$ $%
\in $ $\mathcal{E}_{\mathcal{S}_{\ast }}^{\mathcal{R}}\left( \Omega \right) ,
$ then $\exists N\in \mathcal{R},$ $\forall \beta \in $ $\mathbb{Z}_{+}^{n},$%
\begin{equation*}
\underset{x\in \Omega }{\sup }\left\vert x^{2\beta }u_{\epsilon }\left(
x\right) \right\vert =O\left( \epsilon ^{-N_{2\left\vert \beta \right\vert
}}\right) \text{ },\text{ }\epsilon \longrightarrow 0
\end{equation*}%
From $\left( \ref{Ns Garetto}\right) ,$ we have
\begin{equation*}
\forall m\in \mathbb{Z}_{+},\underset{x\in \Omega }{\sup }\left\vert
u_{\epsilon }\left( x\right) \right\vert =O\left( \epsilon
^{2m+N_{2\left\vert \beta \right\vert }}\right) ,\text{ }\epsilon
\longrightarrow 0
\end{equation*}%
Therefore $\forall x\in \Omega ,$%
\begin{equation*}
\left\vert x^{\beta }u_{\epsilon }\left( x\right) \right\vert
^{2}=\left\vert x^{2\beta }u_{\epsilon }\left( x\right) \right\vert
\left\vert u_{\epsilon }\left( x\right) \right\vert =O\left( \epsilon
^{2m}\right) ,\text{ }\epsilon \longrightarrow 0,
\end{equation*}%
hence
\begin{equation*}
\forall m\in \mathbb{Z}_{+},\left\vert x^{\beta }u_{\epsilon }\left(
x\right) \right\vert =O\left( \epsilon ^{m}\right) ,\text{ }\epsilon
\longrightarrow 0
\end{equation*}
\end{proof}

\begin{definition}
Let $\mathcal{R}$ be a regular subset of $\mathbb{R}_{+}^{\mathbb{Z}_{+}}$,
the algebra of $\mathcal{R-}$regular roughly decreasing generalized
functions, denoted $\mathcal{G}_{\mathcal{S}_{\ast }}^{\mathcal{R}}\left(
\Omega \right) ,$ is the quotient \ algebra
\begin{equation*}
\mathcal{G}_{\mathcal{S}_{\ast }}^{\mathcal{R}}\left( \Omega \right) =\frac{%
\mathcal{E}_{\mathcal{S}_{\ast }}^{\mathcal{R}}\left( \Omega \right) }{%
\mathcal{N}_{\mathcal{S}_{\ast }}\left( \Omega \right) }
\end{equation*}
\end{definition}

\begin{remark}
The $C^{\infty }$\ regularity in the definition of elements of $\mathcal{G}_{%
\mathcal{S}_{\ast }}^{\mathcal{R}}\left( \Omega \right) $\ is not in fact
needed in the proof of the principal results of this work.
\end{remark}

\section{The algebra of $\mathcal{R-}$regular rapidly decreasing generalized
functions}

Let
\begin{equation*}
\mathcal{S}\left( \Omega \right) =\left\{ f\in \mathcal{C}^{\infty }\left(
\Omega \right) :\forall (\alpha ,\beta )\in \mathbb{Z}_{+}^{2n},\underset{%
x\in \Omega }{\sup }\left\vert x^{\beta }\partial ^{\alpha }f\left( x\right)
\right\vert <\infty \right\} ,
\end{equation*}%
called the space of rapidly decreasing functions on $\Omega $, and let $%
\widetilde{\mathcal{R}}$ be a regular subset of $\mathbb{R}_{+}^{\mathbb{Z}%
_{+}^{2}},$ if we define
\begin{equation*}
\mathcal{E}_{\mathcal{S}}^{\widetilde{\mathcal{R}}}\left( \Omega \right)
=\left\{ \left( u_{\epsilon }\right) _{\epsilon }\in \mathcal{S}\left(
\Omega \right) ^{I}:\exists N\in \widetilde{\mathcal{R}},\forall (\alpha
,\beta )\in \mathbb{Z}_{+}^{2n},\underset{x\in \Omega }{\sup }\left\vert
x^{\beta }\partial ^{\alpha }u_{\epsilon }\left( x\right) \right\vert
=O\left( \epsilon ^{-N_{\left\vert \alpha \right\vert ,\left\vert \beta
\right\vert }}\right) ,\epsilon \longrightarrow 0\right\} ,
\end{equation*}%
\begin{equation*}
\mathcal{N}_{\mathcal{S}}^{\widetilde{\mathcal{R}}}\left( \Omega \right)
=\left\{ \left( u_{\epsilon }\right) _{\epsilon }\in \mathcal{S}\left(
\Omega \right) ^{I}:\forall N\in \widetilde{\mathcal{R}},\forall (\alpha
,\beta )\in \mathbb{Z}_{+}^{2n},\underset{x\in \Omega }{\sup }\left\vert
x^{\beta }\partial ^{\alpha }u_{\epsilon }\left( x\right) \right\vert
=O\left( \epsilon ^{N_{\left\vert \alpha \right\vert ,\left\vert \beta
\right\vert }}\right) ,\epsilon \longrightarrow 0\right\} ,
\end{equation*}%
then we have the following results.

\begin{proposition}
\label{prop8}We have the following assertions

$\left( i\right) $ The space $\mathcal{E}_{\mathcal{S}}^{\widetilde{\mathcal{%
R}}}\left( \Omega \right) $ is a subalgebra of $\mathcal{S}\left( \Omega
\right) ^{I}.$

$\left( ii\right) $ The space $\mathcal{N}_{\mathcal{S}}^{\widetilde{%
\mathcal{R}}}\left( \Omega \right) $ is an ideal of $\mathcal{E}_{\mathcal{S}%
}^{\widetilde{\mathcal{R}}}\left( \Omega \right) .$

$\left( iii\right) $ We have $\mathcal{N}_{\mathcal{S}}^{\widetilde{\mathcal{%
R}}}\left( \Omega \right) =\mathcal{N}_{\mathcal{S}}\left( \Omega \right) $,
where
\begin{equation*}
\mathcal{N}_{\mathcal{S}}\left( \Omega \right) =\left\{ \left( u_{\epsilon
}\right) _{\epsilon }\in \mathcal{S}\left( \Omega \right) ^{I}:\forall m\in
\mathbb{Z}_{+},\forall (\alpha ,\beta )\in \mathbb{Z}_{+}^{2n},\underset{%
x\in \Omega }{\sup }\left\vert x^{\beta }\partial ^{\alpha }u_{\epsilon
}\left( x\right) \right\vert =O\left( \epsilon ^{m}\right) ,\epsilon
\longrightarrow 0\right\}
\end{equation*}
\end{proposition}

\begin{proof}
The proof is not difficult, it follows from the properties of the set $%
\widetilde{\mathcal{R}}$.
\end{proof}

\begin{definition}
Let $\widetilde{\mathcal{R}}$ be a regular subset of $\mathbb{R}_{+}^{%
\mathbb{Z}_{+}^{2}},$ the algebra of $\widetilde{\mathcal{R}}\mathcal{-}$%
regular rapidly decreasing generalized functions on $\Omega $, denoted by $%
\mathcal{G}_{\mathcal{S}}^{\widetilde{\mathcal{R}}}\left( \Omega \right) ,$
is the quotient algebra
\begin{equation*}
\mathcal{G}_{\mathcal{S}}^{\widetilde{\mathcal{R}}}\left( \Omega \right) =%
\frac{\mathcal{E}_{\mathcal{S}}^{\widetilde{\mathcal{R}}}\left( \Omega
\right) }{\mathcal{N}_{\mathcal{S}}\left( \Omega \right) }
\end{equation*}
\end{definition}

\begin{example}
$\left( i\right) $ For $\widetilde{\mathcal{R}}=\mathbb{R}_{+}^{\mathbb{Z}%
_{+}^{2}},$ we obtain the algebra $\mathcal{G}_{\mathcal{S}}\left( \Omega
\right) $ of rapidly decreasing generalized functions on $\Omega $, see \cite%
{Grosser}.

$\left( ii\right) $ Pour $\widetilde{\mathcal{R}}=\widetilde{\mathcal{B}},$
we obtain the algebra $\mathcal{G}_{\mathcal{S}}^{\infty }\left( \Omega
\right) $ of regular rapidly decreasing generalized functions on $\Omega ,$
see \cite{Garetto}.
\end{example}

\section{Characterization of $\mathcal{R-}$regular rapidly decreasing
generalized functions}

Let us mention that the theorem of \cite{SY} can be extended to an open
subset $\Omega $ of $\mathbb{R}^{n}$ provided $\Omega $ is a box.

\begin{proposition}
If $\Omega $ is a box of $\mathbb{R}^{n}$, then
\begin{equation}
\mathcal{S}\left( \Omega \right) \mathcal{=S}^{\ast }\left( \Omega \right)
\cap \mathcal{S}_{\ast }\left( \Omega \right) \text{ }  \label{(*)}
\end{equation}
\end{proposition}

\begin{proof}
The proof is the same as in \cite{SY}, noting that, in Taylor's expansion,
the hypothesis that $\Omega $ is a box assures that $\left(
x_{1}+h,x^{\prime }\right) $ stays in $\Omega $ for all $\left(
x_{1},x^{\prime }\right) \in \Omega $ and $h>0$ sufficiently small$.$
\end{proof}

The principal result of this section is an extension of (\ref{(*)})\ to the
algebra of $\widetilde{\mathcal{R}}$-regular rapidly decreasing generalized
functions. It is the first characterization of the algebra $\mathcal{G}_{%
\mathcal{S}}^{\widetilde{\mathcal{R}}}\left( \Omega \right) $.

\begin{theorem}
\label{th10}If $\Omega $ is a box, then
\begin{equation*}
\mathcal{G}_{\mathcal{S}}^{\widetilde{\mathcal{R}}}\left( \Omega \right) =%
\mathcal{G}_{\mathcal{S}_{\ast }}^{\mathcal{R}_{0}}\left( \Omega \right)
\cap \mathcal{G}_{\mathcal{S}^{\ast }}^{\mathcal{R}^{0}}\left( \Omega \right)
\end{equation*}
\end{theorem}

\begin{proof}
We have to show that $\mathcal{E}_{\mathcal{S}}^{\widetilde{\mathcal{R}}%
}\left( \Omega \right) =\mathcal{E}_{\mathcal{S}_{\ast }}^{\mathcal{R}%
_{0}}\left( \Omega \right) \cap \mathcal{E}_{\mathcal{S}^{\ast }}^{\mathcal{R%
}^{0}}\left( \Omega \right) $ and $\mathcal{N}_{\mathcal{S}}\left( \Omega
\right) =\mathcal{N}_{\mathcal{S}_{\ast }}\left( \Omega \right) \cap
\mathcal{N}_{\mathcal{S}^{\ast }}\left( \Omega \right) .$ The inclusions $%
\mathcal{E}_{\mathcal{S}}^{\widetilde{\mathcal{R}}}\left( \Omega \right)
\subset \mathcal{E}_{\mathcal{S}_{\ast }}^{\mathcal{R}_{0}}\left( \Omega
\right) \cap \mathcal{E}_{\mathcal{S}^{\ast }}^{\mathcal{R}^{0}}\left(
\Omega \right) $ and $\mathcal{N}_{\mathcal{S}}\left( \Omega \right) \subset
\mathcal{N}_{\mathcal{S}_{\ast }}\left( \Omega \right) \cap \mathcal{N}_{%
\mathcal{S}^{\ast }}\left( \Omega \right) $ are obvious$.$ In order to show
the reverse inclusions, first let $\left( u_{\epsilon }\right) _{\epsilon
}\in $ $\mathcal{E}_{\mathcal{S}_{\ast }}^{\mathcal{R}_{0}}\left( \Omega
\right) \cap \mathcal{E}_{\mathcal{S}^{\ast }}^{\mathcal{R}^{0}}\left(
\Omega \right) $, then $\left( u_{\epsilon }\right) _{\epsilon }\in \mathcal{%
S}^{\ast }\left( \Omega \right) ^{I}\cap \mathcal{S}_{\ast }\left( \Omega
\right) ^{I}=\mathcal{S}\left( \Omega \right) ^{I}.$ In order to show that $%
\left( u_{\epsilon }\right) _{\epsilon }$ satisfies the estimate of $%
\mathcal{E}_{\mathcal{S}}^{\widetilde{\mathcal{R}}}\left( \Omega \right) $,
set $x=\left( x_{1},x^{\prime }\right) \in \mathbf{I}_{1}\mathbf{\times }%
\left( \mathbf{I}_{2}\times \mathbf{I}_{3}\times ...\times \mathbf{I}%
_{n}\right) :=\Omega $ and consider in first the case $x_{1}>0.$ For $h>0$
sufficiently small$,$ a Taylor's expansion of $u_{\epsilon }$ with respect
to $x_{1}$ gives
\begin{equation}
u_{\epsilon }\left( x_{1}+h,x^{\prime }\right) =u_{\epsilon }\left(
x_{1},x^{\prime }\right) +h\partial _{1}u_{\epsilon }\left( x_{1},x^{\prime
}\right) +\frac{h^{2}}{2}\partial _{1}^{2}u_{\epsilon }\left( \xi ,x^{\prime
}\right) ,  \label{22}
\end{equation}%
for $\xi \in \left] x_{1},x_{1}+h\right[ .$ The hypothesis $\left(
u_{\epsilon }\right) _{\epsilon }\in $ $\mathcal{E}_{\mathcal{S}_{\ast }}^{%
\mathcal{R}_{0}}\left( \Omega \right) \cap \mathcal{E}_{\mathcal{S}^{\ast
}}^{\mathcal{R}^{0}}\left( \Omega \right) $ gives
\begin{equation*}
\exists L\in \mathcal{R}_{0},\forall k\in \mathbb{Z}_{+},\underset{x_{1}>0}{%
\sup }\left( 1+\left\vert x\right\vert \right) ^{k}\left\vert u_{\epsilon
}\left( x\right) \right\vert =O\left( \epsilon ^{-L_{k}}\right) ,\epsilon
\longrightarrow 0
\end{equation*}%
\begin{equation*}
\exists M\in \mathcal{R}^{0},\underset{x_{1}>0}{\sup }\left\vert \partial
_{1}^{2}u_{\epsilon }\left( x\right) \right\vert =O\left( \epsilon
^{-M_{2}}\right) ,\epsilon \longrightarrow 0
\end{equation*}%
We have%
\begin{equation*}
\underset{x_{1}>0}{\sup }\left( 1+\left\vert x\right\vert \right)
^{k}\left\vert u_{\epsilon }\left( x_{1}+h,x^{\prime }\right) \right\vert
\leq \underset{x_{1}>0}{\sup }\left( 1+\left\vert \left( x_{1}+h,x^{\prime
}\right) \right\vert \right) ^{k}\left\vert u_{\epsilon }\left(
x_{1}+h,x^{\prime }\right) \right\vert =O\left( \epsilon ^{-L_{k}}\right)
,\epsilon \longrightarrow 0
\end{equation*}%
It follows from $\left( \ref{22}\right) $
\begin{equation*}
\left\vert \partial _{1}u_{\epsilon }\left( x_{1},x^{\prime }\right)
\right\vert \leq \frac{1}{h}\left[ \left\vert u_{\epsilon }\left(
x_{1}+h,x^{\prime }\right) \right\vert +\left\vert u_{\epsilon }\left(
x_{1},x^{\prime }\right) \right\vert \right] +\frac{h}{2}\left\vert \partial
_{1}^{2}u_{\epsilon }\left( \xi ,x^{\prime }\right) \right\vert
\end{equation*}%
Therefore
\begin{equation*}
\underset{x_{1}>0}{\sup }\left( 1+\left\vert x\right\vert \right)
^{k}\left\vert \partial _{1}u_{\epsilon }\left( x\right) \right\vert
=O\left( \epsilon ^{-L_{k}-M_{2}}\right) ,\epsilon \longrightarrow 0
\end{equation*}%
From $\left( \widetilde{R}3\right) $ of definition \ref{Def2} , there exists
$N^{\prime }\in \widetilde{\mathcal{R}}$ such that
\begin{equation*}
L_{k}+M_{2}\leq N_{2,k}^{\prime }\text{ ,}
\end{equation*}%
consequently
\begin{equation*}
\forall \beta \in \mathbb{Z}_{+}^{n},\underset{x_{1}>0}{\sup }\left\vert
x^{\beta }\partial _{1}u_{\epsilon }\left( x\right) \right\vert \leq C%
\underset{x_{1}>0}{\sup }\left( 1+\left\vert x\right\vert \right)
^{\left\vert \beta \right\vert }\left\vert \partial _{1}u_{\epsilon }\left(
x\right) \right\vert =O\left( \epsilon ^{-N_{2,\left\vert \beta \right\vert
}^{\prime }}\right) ,\epsilon \longrightarrow 0
\end{equation*}%
If $x_{1}<0,$ consider $v_{\epsilon }$ such that $v_{\epsilon }\left(
x\right) =u_{\epsilon }\left( -x_{1},x^{\prime }\right) .$ We see that $%
\left( v_{\epsilon }\right) _{\epsilon }\in \mathcal{E}_{\mathcal{S}_{\ast
}}^{\mathcal{R}_{0}}\left( \Omega \right) \cap \mathcal{E}_{\mathcal{S}%
^{\ast }}^{\mathcal{R}^{0}}\left( \Omega \right) $ and consequently the
above arguments give the existence of $N"\in \widetilde{\mathcal{R}}$ such
that
\begin{equation*}
\underset{x_{1}>0}{\sup }\left\vert x^{\beta }\partial _{1}v_{\epsilon
}\left( x\right) \right\vert =\underset{x_{1}<0}{\sup }\left\vert x^{\beta
}\partial _{1}u_{\epsilon }\left( x\right) \right\vert =O\left( \epsilon
^{-N"_{2,\left\vert \beta \right\vert }}\right) ,\epsilon \longrightarrow 0
\end{equation*}%
Now from $\left( \widetilde{R}1\right) $ and $\left( \widetilde{R}2\right) $
of definition \ref{Def2}$,$ there exists $N$ $\in \widetilde{\mathcal{R}}$
such that
\begin{equation*}
\max \left( N_{2,\left\vert \beta \right\vert }^{\prime },N"_{2,\left\vert
\beta \right\vert }\right) \leq N_{1,\left\vert \beta \right\vert }\text{ ,}
\end{equation*}%
so
\begin{equation*}
\underset{x\in \Omega }{\sup }\left\vert x^{\beta }\partial _{1}u_{\epsilon
}\left( x\right) \right\vert =O\left( \epsilon ^{-N_{1,\left\vert \beta
\right\vert }}\right) ,\epsilon \longrightarrow 0
\end{equation*}%
Analogously we show
\begin{equation*}
\exists N\in \widetilde{\mathcal{R}};\forall \beta \in \mathbb{Z}_{+}^{n},%
\underset{x\in \Omega }{\sup }\left\vert x^{\beta }\partial _{i}u_{\epsilon
}\left( x\right) \right\vert =O\left( \epsilon ^{-N_{1,\left\vert \beta
\right\vert }}\right) ,i=2,...,n
\end{equation*}%
Therefore, by induction, we obtain
\begin{equation*}
\exists N\in \widetilde{\mathcal{R}};\forall \alpha \in \mathbb{Z}%
_{+}^{n},\forall \beta \in \mathbb{Z}_{+}^{n},\underset{x\in \Omega }{\sup }%
\left\vert x^{\beta }\partial ^{\alpha }u_{\epsilon }\left( x\right)
\right\vert =O\left( \epsilon ^{-N_{\left\vert \alpha \right\vert
,\left\vert \beta \right\vert }}\right) ,\epsilon \longrightarrow 0,
\end{equation*}%
i.e. $\left( u_{\epsilon }\right) _{\epsilon }\in \mathcal{E}_{\mathcal{S}}^{%
\widetilde{\mathcal{R}}}\left( \Omega \right) .$

Suppose now that $\left( u_{\epsilon }\right) _{\epsilon }\in \mathcal{N}_{%
\mathcal{S}_{\ast }}\left( \Omega \right) \cap \mathcal{N}_{\mathcal{S}%
^{\ast }}\left( \Omega \right) ,$ then
\begin{equation*}
\forall m\in \mathbb{Z}_{+},\forall k\in \mathbb{Z}_{+},\underset{x_{1}>0}{%
\sup }\left( 1+\left\vert x\right\vert \right) ^{k}\left\vert u_{\epsilon
}\left( x\right) \right\vert =O\left( \epsilon ^{\frac{m}{2}}\right)
,\epsilon \longrightarrow 0
\end{equation*}%
\begin{equation*}
\forall m\in \mathbb{Z}_{+};\underset{x_{1}>0}{\sup }\left\vert \partial
_{1}^{2}u_{\epsilon }\left( x\right) \right\vert =O\left( \epsilon ^{\frac{m%
}{2}}\right) ,\epsilon \longrightarrow 0
\end{equation*}%
We have%
\begin{equation*}
\underset{x_{1}>0}{\sup }\left( 1+\left\vert x\right\vert \right)
^{k}\left\vert u_{\epsilon }\left( x_{1}+h,x^{\prime }\right) \right\vert
\leq \underset{x_{1}>0}{\sup }\left( 1+\left\vert \left( x_{1}+h,x^{\prime
}\right) \right\vert \right) ^{k}\left\vert u_{\epsilon }\left(
x_{1}+h,x^{\prime }\right) \right\vert =O\left( \epsilon ^{\frac{m}{2}%
}\right) ,\epsilon \longrightarrow 0
\end{equation*}%
It follows from $\left( \ref{22}\right) $%
\begin{equation*}
\underset{x_{1}>0}{\sup }\left( 1+\left\vert x\right\vert \right)
^{k}\left\vert \partial _{1}u_{\epsilon }\left( x\right) \right\vert
=O\left( \epsilon ^{m}\right) ,\epsilon \longrightarrow 0
\end{equation*}%
Consequently
\begin{equation*}
\forall m\in \mathbb{Z}_{+},\forall \beta \in \mathbb{Z}_{+}^{n},\underset{%
x_{1}>0}{\sup }\left\vert x^{\beta }\partial _{1}u_{\epsilon }\left(
x\right) \right\vert ^{2}\leq C_{1}\underset{x_{1}>0}{\sup }\left(
1+\left\vert x\right\vert \right) ^{\left\vert \beta \right\vert }\left\vert
\partial _{1}u_{\epsilon }\left( x\right) \right\vert =O\left( \epsilon
^{m}\right) ,\epsilon \longrightarrow 0
\end{equation*}%
If $x_{1}<0,$ consider $v_{\epsilon }$ such that $v_{\epsilon }\left(
x\right) =u_{\epsilon }\left( -x_{1},x^{\prime }\right) $ as above, then we
obtain
\begin{equation*}
\underset{x_{1}>0}{\sup }\left\vert x^{\beta }\partial _{1}v_{\epsilon
}\left( x\right) \right\vert ^{2}=\underset{x_{1}<0}{\sup }\left\vert
x^{\beta }\partial _{1}u_{\epsilon }\left( x\right) \right\vert ^{2}=O\left(
\epsilon ^{m}\right) ,\epsilon \longrightarrow 0
\end{equation*}%
Therefore, by induction, we have
\begin{equation*}
\forall m\in \mathbb{Z}_{+},\forall \alpha \in \mathbb{Z}_{+}^{n},\forall
\beta \in \mathbb{Z}_{+}^{n},\underset{x\in \Omega }{\sup }\left\vert
x^{\beta }\partial ^{\alpha }u_{\epsilon }\left( x\right) \right\vert
=O\left( \epsilon ^{m}\right) ,\epsilon \longrightarrow 0
\end{equation*}%
Thus $\mathcal{N}_{\mathcal{S}_{\ast }}\left( \Omega \right) \cap \mathcal{N}%
_{\mathcal{S}^{\ast }}\left( \Omega \right) \subset \mathcal{N}_{\mathcal{S}%
}\left( \Omega \right) $ and consequently $\mathcal{G}_{\mathcal{S}}^{%
\widetilde{\mathcal{R}}}\left( \Omega \right) =\mathcal{G}_{\mathcal{S}%
_{\ast }}^{\mathcal{R}_{0}}\left( \Omega \right) \cap \mathcal{G}_{\mathcal{S%
}^{\ast }}^{\mathcal{R}^{0}}\left( \Omega \right) .$
\end{proof}

The propositions \ref{propnull1} and \ref{prop.null} give the following
result which characterizes the negligible elements of the algebra $\mathcal{G%
}_{\mathcal{S}}^{\widetilde{\mathcal{R}}}\left( \Omega \right) .$

\begin{corollary}
If $\Omega $ is a box, then an element $\left( u_{\epsilon }\right)
_{\epsilon }$ $\in \mathcal{E}_{\mathcal{S}}^{\widetilde{\mathcal{R}}}\left(
\Omega \right) $ is in $\mathcal{N}_{\mathcal{S}}\left( \Omega \right) $ if
and only if the following condition is satisfied
\begin{equation*}
\forall m\in \mathbb{Z}_{+},\underset{x\in \Omega }{\sup }\left\vert
u_{\epsilon }\left( x\right) \right\vert =O\left( \epsilon ^{m}\right)
,\epsilon \longrightarrow 0
\end{equation*}
\end{corollary}

\section{Characterization of $\mathcal{R-}$regular rapidly decreasing
generalized functions via Fourier transform}

The direct Fourier transform of $u\in \mathcal{S}$, denoted $\widehat{u},$
is defined by
\begin{equation*}
\widehat{u}\left( \xi \right) =\left( 2\pi \right)
^{-\frac{n}{2}}\int e^{-ix\xi }u\left( x\right) dx
\end{equation*}

\begin{definition}
The Fourier transform of $u=\left[ \left( u_{\epsilon }\right) _{\epsilon }%
\right] \in \mathcal{G}_{\mathcal{S}}^{\widetilde{\mathcal{R}}}$, denoted by
$\mathcal{F}_{\mathcal{S}}\left( u\right) $, is defined by
\begin{equation*}
\mathcal{F}_{\mathcal{S}}\left( u\right) =\widehat{u}=\left[ \left( \widehat{%
u_{\epsilon }}\right) _{\epsilon }\right] \text{ in }\mathcal{G}_{\mathcal{S}%
}^{\widetilde{\mathcal{R}}}
\end{equation*}
\end{definition}

\begin{remark}
The inverse Fourier transform of $u\in \mathcal{S}$, denoted $\widetilde{u},$
and the map $\mathcal{F}_{\mathcal{S}}^{-1}$ are defined as usually and in
the same way.
\end{remark}

The following proposition gives one of the main results of the Fourier
transform $\mathcal{F}_{\mathcal{S}}$ and is easy to prove.

\begin{proposition}
The map
\begin{equation*}
\mathcal{F}_{\mathcal{S}}:\mathcal{G}_{\mathcal{S}}^{\widetilde{\mathcal{R}}%
}\longrightarrow \mathcal{G}_{\mathcal{S}}^{\widetilde{\mathcal{R}}}
\end{equation*}
is an algebraic isomorphism.
\end{proposition}

Let
\begin{equation*}
\widehat{\mathcal{S}^{\ast }}=\left\{ f\in \mathcal{C}^{\infty }:\forall
\beta \in \mathbb{Z}_{+}^{n},\underset{\xi \in \mathbb{R}^{n}}{\sup }%
\left\vert \xi ^{\beta }\widehat{f}\left( \xi \right) \right\vert <\infty
\right\} ,
\end{equation*}%
and let $\widetilde{\mathcal{R}}$ be a regular subset of $\mathbb{R}_{+}^{%
\mathbb{Z}_{+}^{2}},$ if we define
\begin{equation*}
\mathcal{E}_{\widehat{\mathcal{S}^{\ast }}}^{\mathcal{R}^{0}}=\left\{ \left(
u_{\epsilon }\right) _{\epsilon }\in \widehat{\mathcal{S}^{\ast }}%
^{I}:\exists N\in \mathcal{R}^{0},\forall \beta \in \mathbb{Z}_{+}^{n},%
\underset{\xi \in \mathbb{R}^{n}}{\sup }\left\vert \xi ^{\beta }\widehat{%
u_{\epsilon }}\left( \xi \right) \right\vert =O\left( \epsilon
^{-N_{\left\vert \beta \right\vert }}\right) ,\epsilon \longrightarrow
0\right\} ,
\end{equation*}%
\begin{equation*}
\mathcal{N}_{\widehat{\mathcal{S}^{\ast }}}^{\mathcal{R}^{0}}=\left\{ \left(
u_{\epsilon }\right) _{\epsilon }\in \widehat{\mathcal{S}^{\ast }}%
^{I}:\forall N\in \mathcal{R}^{0},\forall \beta \in \mathbb{Z}_{+}^{n},%
\underset{\xi \in \mathbb{R}^{n}}{\sup }\left\vert \xi ^{\beta }\widehat{%
u_{\epsilon }}\left( \xi \right) \right\vert =O\left( \epsilon
^{N_{\left\vert \beta \right\vert }}\right) ,\epsilon \longrightarrow
0\right\} ,
\end{equation*}%
then the following proposition is easy to prove.

\begin{proposition}
$\left( i\right) $ The space $\mathcal{E}_{\widehat{\mathcal{S}^{\ast }}}^{%
\mathcal{R}^{0}}$ is a subalgebra of $\widehat{\mathcal{S}^{\ast }}^{I}.$

$\left( ii\right) $ The space $\mathcal{N}_{\widehat{\mathcal{S}^{\ast }}}^{%
\mathcal{R}^{0}}$ is an ideal of $\mathcal{E}_{\widehat{\mathcal{S}^{\ast }}%
}^{\mathcal{R}^{0}}.$

$(iii)$\ The ideal $\mathcal{N}_{\widehat{\mathcal{S}^{\ast }}}^{\mathcal{R}%
^{0}}=\mathcal{N}_{\widehat{\mathcal{S}^{\ast }}}$ , where
\begin{equation*}
\mathcal{N}_{\widehat{\mathcal{S}^{\ast }}}:=\left\{ \left( u_{\epsilon
}\right) _{\epsilon }\in \widehat{\mathcal{S}^{\ast }}^{I}:\forall m\in
\mathbb{Z},\forall \beta \in \mathbb{Z}_{+}^{n},\underset{\xi \in \mathbb{R}%
^{n}}{\sup }\left\vert \xi ^{\beta }\widehat{u_{\epsilon }}\left( \xi
\right) \right\vert =O\left( \epsilon ^{m}\right) ,\epsilon \longrightarrow
0\right\}
\end{equation*}
\end{proposition}

The following proposition characterizes $\mathcal{N}_{\widehat{\mathcal{S}%
^{\ast }}}.$

\begin{proposition}
\label{propos}Let $\left( u_{\epsilon }\right) _{\epsilon }$ $\in \mathcal{E}%
_{\widehat{\mathcal{S}^{\ast }}}^{\mathcal{R}^{0}},$ then $\left(
u_{\epsilon }\right) _{\epsilon }$ $\in \mathcal{N}_{\widehat{\mathcal{S}%
^{\ast }}}^{\mathcal{R}^{0}}$ if and only if the following condition is
satisfied
\begin{equation}
\forall m\in \mathbb{Z}_{+},\underset{\xi \in \mathbb{R}^{n}}{\sup }%
\left\vert \widehat{u_{\epsilon }}\left( \xi \right) \right\vert =O\left(
\epsilon ^{m}\right) \text{ },\text{ }\epsilon \longrightarrow 0
\end{equation}
\end{proposition}

\begin{proof}
The proof is similar to that of proposition \ref{prop.null}.
\end{proof}

\begin{definition}
The algebra $\mathcal{G}_{\widehat{\mathcal{S}^{\ast }}}^{\mathcal{R}^{0}}$
is defined as the quotient algebra
\begin{equation*}
\mathcal{G}_{\widehat{\mathcal{S}^{\ast }}}^{\mathcal{R}^{0}}=\frac{\mathcal{%
E}_{\widehat{\mathcal{S}^{\ast }}}^{\mathcal{R}^{0}}}{\mathcal{N}_{\widehat{%
\mathcal{S}^{\ast }}}}
\end{equation*}
\end{definition}

The next theorem is the second characterization of $\mathcal{G}_{\mathcal{S}%
}^{\widetilde{\mathcal{R}}}.$

\begin{theorem}
\label{prop2}We have
\begin{equation*}
\mathcal{G}_{\mathcal{S}}^{\widetilde{\mathcal{R}}}=\mathcal{G}_{\mathcal{S}%
_{\ast }}^{\mathcal{R}_{0}}\cap \mathcal{G}_{\widehat{\mathcal{S}^{\ast }}}^{%
\mathcal{R}^{0}}
\end{equation*}
\end{theorem}

\begin{proof}
Let $\left( u_{\epsilon }\right) _{\epsilon }\in \mathcal{E}_{\widehat{%
\mathcal{S}^{\ast }}}^{\mathcal{R}^{0}}$, then $\exists C>0$\ such that
\begin{eqnarray*}
\int \left\vert x^{\beta }\widehat{u_{\epsilon }}\left( x\right)
\right\vert dx &\leq &C\underset{x\in \mathbb{R}^{n}}{\sup }\left(
1+\left\vert x\right\vert ^{2}\right) ^{n}\left\vert x^{\beta }\widehat{%
u_{\epsilon }}\left( x\right) \right\vert  \\
&=&O\left( \epsilon ^{-N_{\left\vert \beta \right\vert +2n}^{\prime
}}\right) ,\text{ }\epsilon \longrightarrow 0 \\
&=&O\left( \epsilon ^{-N_{\left\vert \beta \right\vert }}\right) ,\text{ }%
\epsilon \longrightarrow 0,
\end{eqnarray*}%
for some $N\in \mathcal{R}^{0}$. The continuity of $\mathcal{F}$ from $%
\mathbb{L}^{1}$ to $\mathbb{L}^{\infty }$ gives
\begin{equation*}
\left\vert \left\vert \partial ^{\beta }u_{\epsilon }\right\vert \right\vert
_{\mathbb{L}^{\infty }}=O\left( \epsilon ^{-N_{\left\vert \beta \right\vert
}}\right) ,\text{ }\epsilon \longrightarrow 0,
\end{equation*}%
which shows that $\left( u_{\epsilon }\right) _{\epsilon }\in \mathcal{E}_{%
\mathcal{S}^{\star }}^{\mathcal{R}^{0}}$ and therefore $\mathcal{E}_{%
\widehat{\mathcal{S}^{\ast }}}^{\mathcal{R}^{0}}\subset \mathcal{E}_{%
\mathcal{S}^{\ast }}^{\mathcal{R}^{0}}.$ Consequently $\mathcal{E}_{\mathcal{%
S}_{\ast }}^{\mathcal{R}_{0}}\cap \mathcal{E}_{\widehat{\mathcal{S}^{\ast }}%
}^{\mathcal{R}^{0}}\subset \mathcal{E}_{\mathcal{S}}^{\widetilde{\mathcal{R}}%
}.$ In order to show the inverse inclusion let us mention, at first, that
from \cite{SY}, we have
\begin{equation*}
\left( u_{\epsilon }\right) _{\epsilon }\in \mathcal{S}^{I}%
\Longleftrightarrow \left( u_{\epsilon }\right) _{\epsilon }\in \mathcal{S}%
_{\ast }^{I}\bigcap \widehat{\mathcal{S}^{\ast }}^{I}
\end{equation*}%
which implies in particular that $\mathcal{S}\subset $ $\widehat{\mathcal{S}%
^{\ast }}.$ On the other hand if $\left( u_{\epsilon }\right) _{\epsilon
}\in \mathcal{E}_{\mathcal{S}}^{\widetilde{\mathcal{R}}}$, then
\begin{eqnarray*}
\int \left\vert \partial ^{\beta }u_{\epsilon }\left( x\right)
\right\vert dx &\leq &C\underset{x\in \mathbb{R}^{n}}{\sup }\left(
1+\left\vert x\right\vert ^{2}\right) ^{n}\left\vert \partial
^{\beta }u_{\epsilon
}\left( x\right) \right\vert  \\
&=&O\left( \epsilon ^{-N_{\left\vert \beta \right\vert ,2n}^{^{\prime
}}}\right) ,\text{ }\epsilon \longrightarrow 0 \\
&=&O\left( \epsilon ^{-N_{\left\vert \beta \right\vert ,0}}\right) ,\text{ }%
\epsilon \longrightarrow 0,
\end{eqnarray*}%
i.e.
\begin{equation*}
\int \left\vert \partial ^{\beta }u_{\epsilon }\left( x\right)
\right\vert
dx=O\left( \epsilon ^{-N_{\left\vert \beta \right\vert }}\right) ,\text{ }%
\epsilon \longrightarrow 0,
\end{equation*}%
for some $N\in \mathcal{R}^{0}.$ The continuity of $\mathcal{F}$ from $%
\mathbb{L}^{1}$ to $\mathbb{L}^{\infty }$ gives
\begin{equation*}
\left\vert \left\vert \xi ^{\beta }\widehat{u_{\epsilon }}\right\vert
\right\vert _{\mathbb{L}^{\infty }}=O\left( \epsilon ^{-N_{\left\vert \beta
\right\vert }}\right) ,\text{ }\epsilon \longrightarrow 0,
\end{equation*}%
which shows that $\left( u_{\epsilon }\right) _{\epsilon }\in \mathcal{E}_{%
\widehat{\mathcal{S}^{\ast }}}^{\mathcal{R}^{0}}$ and consequently $\left(
u_{\epsilon }\right) _{\epsilon }\in \mathcal{E}_{\mathcal{S}_{\ast }}^{%
\mathcal{R}_{0}}\cap \mathcal{E}_{\widehat{\mathcal{S}^{\ast }}}^{\mathcal{R}%
^{0}}.$ Thus $\mathcal{E}_{\mathcal{S}}^{\widetilde{\mathcal{R}}}\subset
\mathcal{E}_{\mathcal{S}_{\ast }}^{\mathcal{R}_{0}}\cap \mathcal{E}_{%
\widehat{\mathcal{S}^{\ast }}}^{\mathcal{R}^{0}},$ it follows that $\mathcal{%
E}_{\mathcal{S}}^{\widetilde{\mathcal{R}}}=\mathcal{E}_{\mathcal{S}_{\ast
}}^{\mathcal{R}_{0}}\cap \mathcal{E}_{\widehat{\mathcal{S}^{\ast }}}^{%
\mathcal{R}^{0}}.$ A similar proof shows that $\mathcal{N}_{\mathcal{S}}=%
\mathcal{N}_{\mathcal{S}_{\ast }}\bigcap \mathcal{N}_{\widehat{\mathcal{S}%
^{\ast }}}$. Therefore $\mathcal{G}_{\mathcal{S}}^{\widetilde{\mathcal{R}}}=%
\mathcal{G}_{\mathcal{S}_{\ast }}^{\mathcal{R}_{0}}\cap \mathcal{G}_{%
\widehat{\mathcal{S}^{\ast }}}^{\mathcal{R}^{0}}$.
\end{proof}

The following corollary gives a second characterization of the space $%
\mathcal{N}_{\mathcal{S}}.$

\begin{corollary}
An element $\left( u_{\epsilon }\right) _{\epsilon }$ $\in \mathcal{E}_{%
\mathcal{S}}^{\widetilde{\mathcal{R}}}$ is in $\mathcal{N}_{\mathcal{S}}$ if
and only if the following condition is satisfied
\begin{equation}
\forall m\in \mathbb{Z}_{+},\underset{\xi \in \mathbb{R}^{n}}{\sup }%
\left\vert \widehat{u_{\epsilon }}\left( \xi \right) \right\vert =O\left(
\epsilon ^{m}\right) ,\epsilon \longrightarrow 0  \label{null f}
\end{equation}
\end{corollary}

\section{Consequences}

We know that when $\widetilde{\mathcal{R}}=\mathbb{R}_{+}^{\mathbb{Z}%
_{+}^{2}}$ we obtain $\mathcal{G}_{\mathcal{S}}^{\mathbb{R}_{+}^{\mathbb{Z}%
_{+}^{2}}}=\mathcal{G}_{\mathcal{S}}.$ Theorem \ref{th10} gives the
following corollary which is a characterization of the algebra of rapidly
decreasing generalized functions.

\begin{corollary}
We have
\begin{equation*}
\mathcal{G}_{\mathcal{S}}=\mathcal{G}_{\mathcal{S}^{\ast }}\cap \mathcal{G}_{%
\mathcal{S}_{\ast }}\text{ ,}
\end{equation*}
where
\begin{equation*}
\mathcal{G}_{\mathcal{S}^{\ast }}:=\frac{\left\{ \left( u_{\epsilon }\right)
_{\epsilon }\in \mathcal{S}^{\ast I}:\forall \alpha \in \mathbb{Z}%
_{+}^{n},\exists m\in \mathbb{Z}_{+},\underset{x\in \mathbb{R}^{n}}{\sup }%
\left\vert \partial ^{\alpha }u_{\epsilon }\left( x\right) \right\vert
=O\left( \epsilon ^{-m}\right) ,\epsilon \longrightarrow 0\right\} }{\left\{
\left( u_{\epsilon }\right) _{\epsilon }\in \mathcal{S}^{\ast I}:\forall
\alpha \in \mathbb{Z}_{+}^{n},\forall m\in \mathbb{Z}_{+},\underset{x\in
\mathbb{R}^{n}}{\sup }\left\vert \partial ^{\alpha }u_{\epsilon }\left(
x\right) \right\vert =O\left( \epsilon ^{m}\right) ,\epsilon \longrightarrow
0\right\} }\text{ ,}
\end{equation*}
and
\begin{equation*}
\mathcal{G}_{\mathcal{S}_{\ast }}:=\frac{\left\{ \left( u_{\epsilon }\right)
_{\epsilon }\in \mathcal{S}_{\ast }^{I}:\forall \beta \in \mathbb{Z}%
_{+}^{n},\exists m\in \mathbb{Z}_{+},\underset{x\in \mathbb{R}^{n}}{\sup }%
\left\vert x^{\beta }u_{\epsilon }\left( x\right) \right\vert =O\left(
\epsilon ^{-m}\right) ,\epsilon \longrightarrow 0\right\} }{\left\{ \left(
u_{\epsilon }\right) _{\epsilon }\in \mathcal{S}_{\ast }^{I}:\forall \beta
\in \mathbb{Z}_{+}^{n},\forall m\in \mathbb{Z}_{+},\underset{x\in \mathbb{R}%
^{n}}{\sup }\left\vert x^{\beta }u_{\epsilon }\left( x\right) \right\vert
=O\left( \epsilon ^{m}\right) ,\epsilon \longrightarrow 0\right\} }
\end{equation*}
\end{corollary}

We have also the following corollary which is an other characterization of
the algebra $\mathcal{G}_{\mathcal{S}}$.

\begin{corollary}
\label{th2}We have
\begin{equation*}
\mathcal{G}_{\mathcal{S}}=\mathcal{G}_{\mathcal{S}_{\ast }}\cap \mathcal{G}_{%
\widehat{\mathcal{S}^{\ast }}}\text{ ,}
\end{equation*}
where
\begin{equation*}
\mathcal{G}_{\widehat{\mathcal{S}^{\ast }}}:=\frac{\left\{ \left(
u_{\epsilon }\right) _{\epsilon }\in \widehat{\mathcal{S}^{\ast }}%
^{I}:\forall \beta \in \mathbb{Z}_{+}^{n},\exists m\in \mathbb{Z}_{+},%
\underset{\xi \in \mathbb{R}^{n}}{\sup }\left\vert \xi ^{\beta }\widehat{%
u_{\epsilon }}\left( \xi \right) \right\vert =O\left( \epsilon ^{-m}\right)
,\epsilon \longrightarrow 0\right\} }{\left\{ \left( u_{\epsilon }\right)
_{\epsilon }\in \widehat{\mathcal{S}^{\ast }}^{I}:\forall \beta \in \mathbb{Z%
}_{+}^{n},\forall m\in \mathbb{Z}_{+},\underset{\xi \in \mathbb{R}^{n}}{\sup
}\left\vert \xi ^{\beta }\widehat{u_{\epsilon }}\left( \xi \right)
\right\vert =O\left( \epsilon ^{m}\right) ,\epsilon \longrightarrow
0\right\} }
\end{equation*}
\end{corollary}

We also know that when $\widetilde{\mathcal{R}}=\widetilde{\mathcal{B}}$ we
obtain\bigskip\ $\mathcal{G}_{\mathcal{S}}^{\widetilde{\mathcal{B}}}=%
\mathcal{G}_{\mathcal{S}}^{\infty }.$ The next result, which is a corollary
of theorem \ref{th10} for $\mathcal{R}=$ $\widetilde{\mathcal{B}}$ , gives a
characterization of $\mathcal{G}_{\mathcal{S}}^{\infty }.$

\begin{corollary}
\label{prop1}We have
\begin{equation*}
\mathcal{G}_{\mathcal{S}}^{\infty }=\mathcal{G}_{\mathcal{S}^{\ast
}}^{\infty }\cap \mathcal{G}_{\mathcal{S}_{\ast }}^{\infty }\text{ ,}
\end{equation*}
where
\begin{equation*}
\mathcal{G}_{\mathcal{S}^{\ast }}^{\infty }:=\frac{\left\{ \left(
u_{\epsilon }\right) _{\epsilon }\in \mathcal{S}^{\ast I}:\exists m\in
\mathbb{Z}_{+},\forall \alpha \in \mathbb{Z}_{+}^{n},\underset{x\in \mathbb{R%
}^{n}}{\sup }\left\vert \partial ^{\alpha }u_{\epsilon }\left( x\right)
\right\vert =O\left( \epsilon ^{-m}\right) ,\epsilon \longrightarrow
0\right\} }{\left\{ \left( u_{\epsilon }\right) _{\epsilon }\in \mathcal{S}%
^{\ast I}:\forall m\in \mathbb{Z}_{+},\forall \alpha \in \mathbb{Z}_{+}^{n},%
\underset{x\in \mathbb{R}^{n}}{\sup }\left\vert \partial ^{\alpha
}u_{\epsilon }\left( x\right) \right\vert =O\left( \epsilon ^{m}\right)
,\epsilon \longrightarrow 0\right\} }\text{ ,}
\end{equation*}
and
\begin{equation*}
\mathcal{G}_{\mathcal{S}_{\ast }}^{\infty }:=\frac{\left\{ \left(
u_{\epsilon }\right) _{\epsilon }\in \mathcal{S}_{\ast }^{I}:\exists m\in
\mathbb{Z}_{+},\forall \beta \in \mathbb{Z}_{+}^{n},\underset{x\in \mathbb{R}%
^{n}}{\sup }\left\vert x^{\beta }u_{\epsilon }\left( x\right) \right\vert
=O\left( \epsilon ^{-m}\right) ,\epsilon \longrightarrow 0\right\} }{\left\{
\left( u_{\epsilon }\right) _{\epsilon }\in \mathcal{S}_{\ast }^{I}:\forall
m\in \mathbb{Z}_{+},\forall \beta \in \mathbb{Z}_{+}^{n},\underset{x\in
\mathbb{R}^{n}}{\sup }\left\vert x^{\beta }u_{\epsilon }\left( x\right)
\right\vert =O\left( \epsilon ^{m}\right) ,\epsilon \longrightarrow
0\right\} }
\end{equation*}
\end{corollary}

We have also the following result obtained as a corollary of the theorem \ref%
{prop2}.

\begin{corollary}
We have
\begin{equation*}
\mathcal{G}_{\mathcal{S}}^{\infty }=\mathcal{G}_{\mathcal{S}_{\ast
}}^{\infty }\cap \mathcal{G}_{\widehat{\mathcal{S}^{\ast }}}^{\infty }
\end{equation*}%
where
\begin{equation*}
\mathcal{G}_{\widehat{\mathcal{S}^{\ast }}}^{\infty }:=\frac{\left\{ \left(
u_{\epsilon }\right) _{\epsilon }\in \widehat{\mathcal{S}^{\ast }}%
^{I}:\exists m\in \mathbb{Z}_{+},\forall \beta \in \mathbb{Z}_{+}^{n},%
\underset{\xi \in \mathbb{R}^{n}}{\sup }\left\vert \xi ^{\beta }\widehat{%
u_{\epsilon }}\left( \xi \right) \right\vert =O\left( \epsilon ^{-m}\right)
,\epsilon \longrightarrow 0\right\} }{\left\{ \left( u_{\epsilon }\right)
_{\epsilon }\in \widehat{\mathcal{S}^{\ast }}^{I}:\forall m\in \mathbb{Z}%
_{+},\forall \beta \in \mathbb{Z}_{+}^{n},\underset{\xi \in \mathbb{R}^{n}}{%
\sup }\left\vert \xi ^{\beta }\widehat{u_{\epsilon }}\left( \xi \right)
\right\vert =O\left( \epsilon ^{m}\right) ,\epsilon \longrightarrow
0\right\} }
\end{equation*}
\end{corollary}

\end{document}